\documentclass[12pt]{amsart}

\usepackage{
amsfonts,
latexsym,
amssymb,
}

\newtheorem{theorem}{Theorem}
\newtheorem{lemma}[theorem]{Lemma}

\newtheorem{proposition}[theorem]{Proposition} 
 
\newtheorem{corollary}[theorem]{Corollary}

\newtheorem*{theorem*}{Theorem}
\newtheorem*{corollary*}{Corollary}

\theoremstyle{definition}

\theoremstyle{remark}

\newcommand{\brfrt}{\hspace{0 pt}}

\DeclareMathOperator{\CAP}{CAP}
\DeclareMathOperator{\cf}{cf}
\newcommand{\iit}{\mathfrak{iit}}

\newcommand{\m}{\mathfrak}

\begin{document} 

\title
[Compactness with respect to a set of filters] 
{Topological spaces compact with respect to a set of filters}

\author{Paolo~Lipparini}

\address{
Dipartimento di Matematica\\II Universit\`a di Roma (Tor Vergata)\\Viale della Ricerca Scientifika, I-00133 ROME ITALY
         } 

\begin{abstract}
If $\mathcal P$ is a family of filters over some set $I$,
a topological space  $X$ is \emph{sequencewise $\mathcal P$-\brfrt compact}
if, for every $I$-indexed sequence of elements of $X$,
there is $F \in \mathcal P$  such that the sequence 
has an $F$-limit point. Countable compactness, sequential compactness, initial $\kappa$-compactness, $[ \lambda ,\mu]$-compactness, the Menger 
and Rothberger properties can all  be expressed in terms
of sequencewise $\mathcal P$-\brfrt compactness, for appropriate choices of $\mathcal P$. \\
We show that
sequencewise $\mathcal P$-\brfrt compactness is preserved under taking  products
if and only if there is a filter $F \in \mathcal P$  
such that sequencewise $\mathcal P$-\brfrt compactness is equivalent to   $F$-compactness.
 If this is the case, and there exists a 
sequencewise $\mathcal P$-\brfrt compact
$T_1$ topological space 
with more than one point, then 
$F$  is necessarily an ultrafilter.
The particular cases of sequential compactness 
and of $[ \lambda ,\mu]$-compactness are
 analyzed in detail.
\end{abstract}


\maketitle

\section{Introduction} \label{intro} 

Kombarov \cite{Ko} 
 generalized the notion
of ultrafilter compactness for topological spaces (Bernstein \cite{Be})
by taking into account a family $\mathcal P$  of ultrafilters,
rather than just a single ultrafilter.  
We extend Kombarov 
 notion to  families of filters
(not necessarily maximal). This
provides an essential strengthening:
for example, also sequential compactness and the Rothberger property become
particular cases.

 We assume no separation axiom, if not otherwise specified.
In order to avoid trivial exceptions, all topological spaces under consideration are assumed to be nonempty.

We now recall the main definitions. 
If $X$ is a topological space, 
$I$ is a set, $(x_i) _{i \in I} $ 
is an $I$-indexed sequence  of elements of  $X$,
and $F$  is a filter over $I$, 
 a point $x \in X$ is an
 \emph{$F$-limit point} 
(Choquet \cite{Ch}, Katetov \cite{Kat})
 of the sequence
$(x_i) _{i \in I} $
 if 
 $\{ i \in I \mid x_i \in U\} \in F$,
for every open neighborhood $U$ of $x$.

If $\mathcal P$ is a family of filters over $I$,
we say that $X$ is \emph{sequencewise $\mathcal P$-\brfrt compact}
 if, for every $I$-indexed sequence of elements of $X$,
there is $F \in \mathcal P$  such that the sequence 
has an $F$-limit point.
Kombarov \cite{Ko} introduced the above notion under the name 
\emph{$\mathcal P$-\brfrt 
compactness} in the particular case when 
$\mathcal P$ is a family of non principal ultrafilters over $ \omega$.
As far as we know, for $\mathcal P$  a family of ultrafilters over an arbitrary 
infinite 
cardinal $\alpha$, the notion  has been first considered
in Garc{\'{\i}}a-Ferreira \cite[Definition 1.2(1)]{GFfup} (under the name  
\emph{quasi $\mathcal P$-\brfrt 
compactness}).
We have chosen  the present terminology 
in the hope to avoid any possible ambiguity.
Apparently, in the above general form, the case in which $\mathcal P$ is a family of filters
has never been considered,
before we discussed it in \cite{cmuc},
under different terminology.
Notice that sequencewise $\mathcal P$-\brfrt compactness
is trivially closed hereditary and preserved under surjective continuous
images. 

We now present some examples.
When $\mathcal P= \{ F \} $ is a singleton,
we get the notion of \emph{$F$-compactness},
particularly studied in the case when $F$ is an ultrafilter
\cite{Be,GS,Sa}. 
As another example, a topological space is countably compact if and only if, in 
the present terminology, 
it is sequencewise $\mathcal P$-\brfrt compact, where $\mathcal P$ is the family of 
all uniform ultrafilters over $ I =\omega$ (Ginsburg and Saks \cite[p. 404]{GS}). 
More generally, for $\lambda$ regular,
a topological space satisfies $\CAP _ \lambda $
(every subset of cardinality $\lambda$ has a complete accumulation point)
 if and only if
it is sequencewise $\mathcal P$-\brfrt compact, for  the family $\mathcal P$ of 
all uniform ultrafilters over $ I = \lambda $ (Saks \cite[pp. 80--81]{Sa}). The assumption 
$\lambda$ regular is only for simplicity,  similar results hold for $\lambda$ singular,
and for pairs of cardinals, as well (\cite{Sa}; see also \cite[Sections 3 and 4]{tproc2}).
Caicedo \cite{Ca}  extended some of  the above results
and simplified many arguments; in particular, it 
follows easily from \cite[Section 3]{Ca} 
that a topological space is
 $[ \mu, \lambda ]$-\brfrt compact 
(Alexandroff and Urysohn \cite{AU}, Smirnov \cite{Sm}) 
if and only if it is sequencewise $\mathcal P$-\brfrt compact, for
the family $\mathcal P$ of all $( \mu, \lambda)$-regular ultrafilters
over $ [\lambda ] ^{< \mu} $ (the set of all subsets of $\lambda$
of cardinality $ < \mu$). 
In particular, the above examples include initial $\lambda$-compactness.
Also the Menger, the Rothberger and related properties can be given an equivalent formulation in terms of 
sequencewise $\mathcal P$-compactness. See  \cite{ufmeng}.

As another example, sequential compactness is 
equivalent to sequencewise $\mathcal P$-\brfrt compactness,
for the following choice of $\mathcal P$.
If $Z$ is an infinite subset of $ \omega$, let 
$F_Z = \{ W \subseteq \omega \mid |Z \setminus W| \text{ is finite}\} $,
that is, $F_Z$ is the filter on $ \omega$ generated by the Fr\'echet filter on $Z$. 
We now get sequential compactness  
by taking
$ I=\omega$ and  $\mathcal P = \{ F_Z \mid Z \text{ an infinite subset of $ \omega$}  \} $.

Related subjects have been treated in a great  generality in \cite{cmuc},
where, under different terminology, we showed that a large class
of both covering properties and accumulation properties 
 can be expressed by means of sequencewise $\mathcal P$-\brfrt compactness   (see \cite[Remark 5.4]{cmuc}). Besides the  examples mentioned above, in \cite{cmuc} we also considered various compactness
properties defined in terms of ordinal numbers. 

A brief summary of the paper  follows. In Section \ref{pres} we prove some
theorems on preservation of sequencewise $\mathcal P$-\brfrt compactness under
products. Section \ref{examplessec} discusses many examples,
and shows that some classical results can be obtained as a consequence of the theorems in Section \ref{pres}. 
In Section \ref{lmcpn} we deal with  $[ \mu, \lambda ]$-\brfrt compactness, while
the example of
sequential compactness is dealt with in detail in Section \ref{furth};
there we show 
 that all 
products of members of some family $\mathcal T$ are sequentially compact
if and only if in all members of $\mathcal T$ every sequence converges.  
Finally, the last section contains some additional comments, states some problems, and
introduces a pseudocompact-like generalization. Moreover, the connections with Comfort order on ultrafilters
are briefly discussed.

\section{Preservation under products} \label{pres}

We state our main result 
in a form relative to a class $\mathcal K$ of topological spaces,
since there are significant applications.
For a class $\mathcal K$ of topological spaces, we say that
sequencewise $\mathcal P$-\brfrt compactness  and sequencewise 
$\mathcal P'$-compactness  are \emph{equivalent in} $\mathcal K$
if, for every topological space $X \in \mathcal K$,   
$X$ is sequencewise $\mathcal P$-\brfrt compact if and only if $X$ is sequencewise $\mathcal P'$-compact (in the above definition, all members of $\mathcal P$ are filters over some set $I$, and all members of $\mathcal P'$ are filters over some set $I'$, but we are not necessarily assuming that $I=I'$).

In all product theorems below we allow repetitions, i.~e.,
we allow a space occur multiple times. 

\begin{theorem} \label{prodd}
Suppose that $\mathcal K$ is a class of topological spaces, and
 $\mathcal P$ is a nonempty family of filters over some set $I$.
Then the following conditions are equivalent.
 \begin{enumerate}  
 \item
Every product of sequencewise $\mathcal P$-\brfrt compact spaces which are members of $\mathcal K$
is sequencewise $\mathcal P$-\brfrt compact. 
 \item
Every product of $ |\mathcal P|$ many sequencewise $\mathcal P$-\brfrt compact spaces which are members of $\mathcal K$
is sequencewise $\mathcal P$-\brfrt compact. 
\item
Sequencewise $\mathcal P$-\brfrt compactness is equivalent in $\mathcal K$ to 
$F$-\brfrt compactness, for some filter $F \in \mathcal P$.
  \end{enumerate} 

If the class $\mathcal K$ is preserved under taking products, then the preceding
conditions are also equivalent to:
  \begin{enumerate}    
\item[(4)]
Sequencewise $\mathcal P$-\brfrt compactness is equivalent in $\mathcal K$ to 
$F$-\brfrt compactness, for some filter $F$ over some set $J$.
   \end{enumerate} 

Moreover, if either (3) or (4) above holds, and there exists in $\mathcal K$ some 
sequencewise $\mathcal P$-\brfrt compact
topological space with two disjoint nonempty 
closed sets, 
then any  $F$ as in (3) or (4) is an ultrafilter.
 \end{theorem}  

\begin{proof}
(1) $\Rightarrow $  (2) and (3) $\Rightarrow $  (4) are trivial.

(2) $\Rightarrow $  (3)
Suppose that (3) fails, that is, for no $F \in \mathcal P$,  
sequencewise $\mathcal P$-\brfrt compactness is equivalent in $\mathcal K$ to 
$F$-compactness. We shall find
a family of $|\mathcal P|$ many sequencewise $\mathcal P$-\brfrt compact spaces in $\mathcal K$
whose product is not sequencewise $\mathcal P$-\brfrt compact, thus (2) fails. 
For every 
$F \in \mathcal P$,
$F$-compactness
trivially implies 
sequencewise $\mathcal P$-\brfrt compactness, hence if they are
not equivalent in $\mathcal K$, there is a topological space $X_F \in \mathcal K$
which is sequencewise $\mathcal P$-\brfrt compact but not $F$-compact.
The latter means that  there is a sequence  
$(x _{i,F} ) _{i \in I} $
of elements of $X_F$ 
which has no $F$-limit point in $X_F$.
For every $F \in \mathcal P$, choose some space and some sequence as above, and
let $X= \prod _{F \in \mathcal P} X_F$.
Consider in $X$ the sequence 
$(y _{i} ) _{i \in I} $ 
defined in such a way that, for each $i \in I$,
the projection of $y_i$ on $X_F$ is  
$x _{i,F}  $.
It is a well known fact (see, e.~g., \cite[Theorem 2.1]{Sa}),
 that, for every filter $F$, a sequence in a product has an $F$-limit point
if and only if all of its projections onto each factor have an $F$-limit point.
 Thus, given any $F \in \mathcal P$, the sequence $(y _{i} ) _{i \in I} $ 
cannot have  an $F$-limit point, as witnessed by its projection on $X_F$.
Hence $X$ is not sequencewise $\mathcal P$-compact.

(3) $\Rightarrow $  (1) It is a standard argument \cite{Be} to show that $F$-compactness
is preserved under products. Hence if  
sequencewise $\mathcal P$-\brfrt compactness is equivalent to 
$F$-compactness in $\mathcal K$, for some filter $F$, then 
any product of sequencewise $\mathcal P$-\brfrt compact
spaces which are members of $\mathcal K$ is $F$-compact.
Since $F \in \mathcal K$, 
then $F$-compactness implies  
sequencewise $\mathcal P$-\brfrt compactness, thus 
(1) holds.

The proof that (4) implies  (1), 
under the assumption that $\mathcal K$ is preserved under taking products,
is similar. As above, 
any product of sequencewise $\mathcal P$-\brfrt compact
spaces which are members of $\mathcal K$ is $F$-compact.
By the assumption, any such product is still a member of $\mathcal K$;
then (4) implies that the product is
sequencewise $\mathcal P$-\brfrt compact,
thus (1) holds.

To prove the last statement, suppose that the filter $F$ over, say, $J$ is not an 
ultrafilter. This implies that there are disjoint subsets  $J_1, J_2 \subseteq J$
such that $J_1 \cup  J_2 = J$ and neither $J_1 $ nor $ J_2 $ belongs to $F$. If
 $X $ is a
topological space with two disjoint nonempty closed sets $C_1$, $C_2$, then $X$ is
not $F$-compact, as shown by any sequence 
$(x_j) _{j \in J} $
such that $x_j \in C_1$ for  $j \in J_1$ and 
$x_j \in C_2$ for  $j \in J_2$.
 \end{proof}

In the particular case when all members of $\mathcal K$
are sequencewise $\mathcal P$-\brfrt compact, the statement
of Theorem \ref{prodd} becomes somewhat simpler. 

\begin{corollary} \label{proddd}
Suppose that $\mathcal K$ is a class of topological spaces, 
and  $\mathcal P$ is a nonempty family of filters over some set $I$.
Then the following conditions are equivalent.
 \begin{enumerate}  
 \item
Every product of  members of $\mathcal K$
is sequencewise $\mathcal P$-\brfrt compact. 
 \item
Every product of $ |\mathcal P|$ many  members of $\mathcal K$
is sequencewise $\mathcal P$-\brfrt compact. 
\item
There is some filter $F \in \mathcal P$ such that 
every member of $\mathcal K$ is  
$F$-\brfrt compact.
  \end{enumerate} 
\end{corollary}  

\begin{proof}
Each of the conditions in the corollary implies that 
every member of $\mathcal K$
is sequencewise $\mathcal P$-\brfrt compact. 
Under this assumption, each condition is
equivalent to the respective condition in Theorem \ref{prodd}. 
\end{proof}

\begin{corollary} \label{x}
Suppose that $X$ is a topological space, and
 $\mathcal P$ is a nonempty family of filters over some set $I$.
Then 
every power of $X$ is
 sequencewise $\mathcal P$-\brfrt compact 
 if and only if  $ X ^{|\mathcal P|}$ 
is sequencewise $\mathcal P$-\brfrt compact, 
if and only if 
$X$ is  
$F$-compact, for some  $F \in\mathcal P$.

If the above conditions hold, and $X$ has
two disjoint nonempty 
closed sets, 
then any  $F$ as above is an ultrafilter. 
\end{corollary}

 \begin{proof}
Immediate from Corollary  \ref{proddd},  by taking $\mathcal K = \{ X \} $.
The last statement follows from the last  statement in Theorem \ref{prodd}.
 \end{proof}

\section{Examples} \label{examplessec} 

 Many 
results can be obtained as particular cases of 
Theorem \ref{prodd} and 
Corollary \ref{proddd}. 
For example, by the 
mentioned characterization of 
countable compactness, we get that,
for every class $\mathcal K$, every product of members of $\mathcal K$ 
is countably compact 
if and only if every product of
$2 ^{2 ^ \omega } $ members of $\mathcal K$ 
is countably compact,
if and only if there is an ultrafilter $F$ uniform over $ \omega$
such that every member of $\mathcal K$  is $F$-compact.
For powers of a single space this result is due
 to Ginsburg and Saks \cite[Theorem 2.6]{GS}.  

In a similar way, by applying the techniques of 
\cite[Sections 2 and 6]{Sa},
for every class $\mathcal K$
and every infinite cardinal $\lambda$,
we have that every product of members of $\mathcal K$ 
satisfies $\CAP _ \lambda $
if and only if every product of
$2 ^{2 ^ \lambda  } $ members of $\mathcal K$ satisfies $\CAP _ \lambda $,
if and only if there is an ultrafilter $F$ uniform over $\lambda$ 
such that every member of $\mathcal K$  is $F$-compact.

A completely analogous characterization of those classes $\mathcal K$ 
such that all products of members of $\mathcal K$ 
 are
$[ \mu, \lambda ]$-\brfrt compact
is obtained by using  $( \mu, \lambda)$-regular ultrafilters
over $ [\lambda ] ^{< \mu} $ \cite[Theorem 3.4]{Ca}.
We are going to give an improved version in the next section.

In particular, the above results furnish a characterization
of those classes $\mathcal K$ 
such that all products of members of $\mathcal K$ are
initially $\lambda$-compact. 
By \cite[Theorem 6.2]{Sa}, 
there exists some family $\mathbf P$
such that a topological space is 
 initially $\lambda$-compact
if and only if it is 
sequencewise $\mathcal P$-\brfrt compact,
for all $\mathcal P \in \mathbf P$.
Garc{\'{\i}}a-Ferreira \cite[Corollary 2.15]{GF1} 
improved this to a single family,
and this follows also from \cite{Ca} in a simpler way
(and with an improved bound $2 ^{2^ \lambda } $).
Moreover, in \cite[Theorem 2.17]{GF1}
it is proved that, for any given cardinal $\lambda$, 
if initial $\lambda$-compactness
is preserved under products, then there is some ultrafilter 
$D$ such that 
$D$-compactness is equivalent to 
initial $\lambda$-compactness.
Theorem \ref{prodd}, applied to the case when $\mathcal K$
is the class of all topological spaces, together with 
the characterization of initial $\lambda$-compactness 
in terms of sequencewise $\mathcal P$-\brfrt compactness, furnishes a simpler proof of
\cite[Theorem 2.17]{GF1}; actually, this proof  shows that 
initial $\lambda$-compactness
is preserved under products if and only if  there is some ultrafilter 
$D$ such that 
$D$-compactness is equivalent to 
initial $\lambda$-compactness
(and, if this is the case,  then $D$ can be chosen $( \omega, \lambda )$-regular over $|[ \lambda] ^{< \omega} ] |= \lambda $).
Moreover, all the above arguments apply 
to $[ \mu, \lambda ]$-\brfrt compactness, too; see the next section.
Other results about preservation of 
$[ \mu, \lambda ]$-\brfrt compactness
under products are given in \cite{prodlm},
where we establish a connection with strongly compact
cardinals.

In \cite{ufmeng} we provide characterizations
of the Menger property and of the Rothberger property 
in terms of sequencewise $\mathcal P$-\brfrt compactness.
Actually, we consider even more general properties, which
depend on three cardinals.
Though Theorem \ref{prodd} can be applied in
this situation, too, in \cite[Theorem 2.3]{ufmeng} we are able to 
obtain stronger bounds in a direct way.  
Anyway, the above characterizations are good examples
of the usefulness of allowing non maximal filters in $\mathcal P$;
indeed, the characterization of the Rothberger property
involves a family $\mathcal P$ consisting of filters none
of which is maximal \cite[Proposition 4.1, and the comments below]{ufmeng}.
 Moreover, these examples show how  significant the difference
is
between the case in which $\mathcal P$ contains
some ultrafilter, and the case in which $\mathcal P$ contains no ultrafilter
(compare the last statement in Theorem \ref{prodd}).
Indeed, there are $T_1$ spaces all whose powers are
Menger (they are exactly the compact spaces); on the contrary,
in \cite[Corollary 4.2]{ufmeng} we prove that if some product of
$T_1$ spaces is Rothberger, then all but at most a finite 
number of the factors are one-element. 
A somewhat similar situation  occurs in the case of sequential
compactness, as we are going to discuss in Section \ref{furth}.

\section{$[ \mu, \lambda ]$-compactness} \label{lmcpn} 

In this section we  provide a slight improvement of \cite[Theorem 3.4]{Ca}.
We need an auxiliary proposition, which may have independent interest.
 First, we prove a lemma which is essentially a restatement of well-known facts.
The \emph{initial interval topolgy} $\iit$ on a cardinal $\mu $ is
the topology whose opens sets are the intervals of the form 
$[0, \alpha )$ ($ \alpha \leq \mu$).

\begin{lemma} \label{reg} 
If $\mu $ is a regular cardinal, then a topological space $X$ 
fails to be $[\mu, \mu]$-compact if and only if there is a continuous
surjective function $f: X \to ( \mu, \iit)$. 
\end{lemma} 

\begin{proof}
One implication is trivial, since 
$( \mu, \iit)$ is trivially not 
 $[\mu, \mu]$-compact, $\mu $ being a regular cardinal, and since  $[\mu, \mu]$-compactness
is preserved under continuous surjective images.

For the reverse implication, it is well-known that, for $\mu $ regular,
$X$ fails to be  $[\mu, \mu]$-compact if and only if 
there is a decreasing sequence $(C_ \alpha ) _{ \alpha \in \mu} $
of nonempty closed subsets of $X$ with empty intersection 
(see, e.~g., \cite[Theorem 4.4]{tproc2} for the proof in a more general context).
Define $f: X \to ( \mu, \iit)$ by 
$f(x) = \sup \{ \alpha \in \mu \mid x \in C_ \alpha \} $.
Notice that the range of $f$ is contained in $\mu $, since the sequence $(C_ \alpha ) _{ \alpha \in \mu} $
is decreasing with empty intersection.
Moreover, 
$f$ is continuous, since  
$f ^{-1} ([ \beta , \mu)) = C_ \beta   $,  
if $\beta \in \mu $ is a successor ordinal, 
and 
$f ^{-1} ([ \beta , \mu)) = \bigcap _{ \alpha < \beta }C_ \alpha  $, 
if $\beta \in \mu $ is limit;
hence the counterimage by $f$ of a closed set is closed.
Though $f$ need not be surjective, in general,
the image of $f$  is cofinal in $\mu $, since the $C_ \alpha  $'s
are nonempty but with empty intersection; thus the image of 
$f$ is homeomorphic to  $( \mu, \iit)$, and the lemma is proved. 
 \end{proof} 

The next proposition  does not give the best possible
 bounds, however it is enough for our purposes here; on the other hand,
 a proof of an optimal result would be considerably more involved.

\begin{proposition} \label{sing}
Suppose that  $\lambda$ is a singular cardinal and 
let $\kappa= 2 ^{ \lambda ^{< \lambda } } $.

If some product $\prod _{j \in J} X_j$
is $[ \lambda, \lambda ]$-compact, then
either

$|\{ j \in J \mid X_j \text{ is not $[\cf \lambda , \cf \lambda ]$\brfrt -compact}\}| < \kappa $,  or

$|\{ j \in J \mid X_j \text{ is not } 
[\lambda ^+, \lambda^+ ] \text{\brfrt -compact}\}| < \kappa $.

In particular, if $Y^ \kappa $ is $[ \lambda, \lambda ]$-compact, then
$Y$ is either $[\cf \lambda , \cf \lambda ]$-compact or $[\lambda ^+, \lambda^+ ]$\brfrt -compact.
 \end{proposition}  

\begin{proof} 
Suppose by contradiction that 
$A, B \subseteq J$, 
$| A|= |B|= \kappa $, and that
$X_j $  is not $[\cf \lambda , \cf \lambda ]$-compact, 
for $j \in A$, and not
$[\lambda ^+, \lambda^+ ]$-compact,
for $j \in B$.
It is enough to show that
$\prod _{j \in A \cup B} X_j$
is not $[ \lambda, \lambda ]$-compact.
Moreover, since $\kappa$ is infinite, then, without loss of generality, we can suppose
$A \cap B = \emptyset $. 
By Lemma \ref{reg} applied $\kappa$ times in the case of the regular cardinals
$\cf \lambda $ and $\lambda^+$,  and by naturality of the product,
we have a surjective continuous function from
 $\prod _{j \in A \cup B} X_j$ to 
 $\cf \lambda ^ \kappa \times (\lambda^+) ^{ \kappa } 
\cong (\cf \lambda \times \lambda^+) ^{ \kappa }$,
 where both $\cf \lambda $ and $\lambda^+$ are endowed with 
the $\iit$ topology. 
The very same arguments of \cite[Proposition 5]{prodlm} 
applied to $X= \cf \lambda \times \lambda^+$ show 
that $X^ \kappa $ is not  
$[ \lambda, \lambda ]$-compact, hence
$\prod _{j \in A \cup B} X_j$
is not $[ \lambda, \lambda ]$-compact.
\end{proof}

The next theorem complements \cite[Theorem 3.4]{Ca},
which provides the better bound $ 2 ^{2 ^{ \lambda } } $, but only
in the case when
$\cf \lambda \geq \mu$ (in particular, when $\lambda$ is regular).

\begin{theorem} \label{singca}
Suppose that $\mu \leq \lambda $, $\lambda$ is singular, and $\mathcal T$ is a class of topological spaces. Then
all products of members of $\mathcal T$ are  $[ \mu, \lambda ]$-compact
if and only  all products of $ 2 ^{2 ^{ \lambda ^+} } $ members of $\mathcal T$ are  $[ \mu, \lambda ]$-compact.
 \end{theorem} 

\begin{proof}
One implication is trivial. 
We shall first prove the  converse in the particular case 
$\mu = \lambda $.
Suppose that all products of $ 2 ^{2 ^{ \lambda ^+} } $ members of $\mathcal T$ are  $[ \lambda , \lambda ]$-compact.
From 
Proposition \ref{sing}
we get that either 
(1) all  products of $ 2 ^{2 ^{ \lambda ^+} } $ members of $\mathcal T$
are $[ cf \lambda , \cf \lambda ]$-compact,
or (2) all  products of $ 2 ^{2 ^{ \lambda ^+} } $ members of $\mathcal T$
are  $[ \lambda^+ , \lambda^+ ]$-compact.
Indeed, if  $Y_1 $, $  Y_2$ were, respectively,  
not $[ cf \lambda , \cf \lambda ]$-compact
and not 
$[ \lambda^+ , \lambda^+ ]$-compact 
products of $ 2 ^{2 ^{ \lambda ^+} } $ members of $\mathcal T$,
then, since 
$ \kappa = 2 ^{ \lambda ^{< \lambda } } \leq 2 ^{2 ^{ \lambda ^+} } $,
 also 
$(Y_1 \times Y_2) ^ \kappa $
could be written as 
a product of $ 2 ^{2 ^{ \lambda ^+} } $ members of $\mathcal T$,
henceforth $(Y_1 \times Y_2) ^ \kappa $  would be $[ \lambda , \lambda ]$-compact, but then 
Proposition \ref{sing} gives a contradiction.
In eventuality (1) above, we have  
that all products of members of $\mathcal T$ are 
 $[ \cf\lambda , \cf\lambda ]$-compact, by Caicedo's theorem,
since $ 2 ^{2 ^{ \cf\lambda } } \leq  2 ^{2 ^{ \lambda ^+} } $; then
 the conclusion follows from the easy fact that
 $[ \cf\lambda , \cf\lambda ]$-compactness implies 
 $[ \lambda , \lambda ]$-compactness.
In eventuality (2), in a similar way, we have that
all products of members of $\mathcal T$ are 
$[ \lambda^+ , \lambda^+ ]$-compact.
This implies that 
all products of members of $\mathcal T$ are 
$[ \lambda , \lambda ]$-compact, by another theorem by
Caicedo \cite[Corollary 1.8(ii)]{Ca}.
We have proved the theorem in the particular case $\mu = \lambda $. 

To prove the general case, notice that, since 
$\mu  \leq \lambda $, then
$[ \mu, \lambda ]$-compactness
implies $[ \lambda , \lambda ]$-compactness, hence we get from the above special case that
 all products of members of $\mathcal T$ are 
$[ \lambda , \lambda ]$-compact.
Moreover, for every regular $\lambda' < \lambda $,
all such products are $[ \mu, \lambda' ]$-compact, hence, 
since $ 2 ^{2 ^{ \lambda' } } \leq  2 ^{2 ^{ \lambda ^+} } $, and again using Caicedo's theorem, we get that
all products of members of $\mathcal T$ are  $[ \mu, \lambda' ]$-compact.
Since $\mu < \lambda $ and $\lambda$ is limit, then $[ \mu, \lambda ]$-compactness is equivalent to the conjunction
of $[ \lambda , \lambda ]$-compactness and of $[ \mu, \lambda' ]$-compactness
for every regular $\lambda' < \lambda $, thus 
all products of members of $\mathcal T$ are  $[ \mu, \lambda ]$-compact. 
 \end{proof}

\section{Sequential compactness} \label{furth}

Recall that a space $X$ is called \emph{ultraconnected} if
no pair of nonempty closed sets of $X$  is disjoint. Equivalently,
a space is ultraconnected if and only if 
$ \overline{\{ x_1 \}} \cap \dots \cap \overline{\{ x_n \}} \not= \emptyset  $, 
for every $n>0$ and every $n$-tuple  $x_1$, \dots  $x_n$
of elements of $X$, where overline denotes closure.

We need the following easy lemma, for which we know no reference.

\begin{lemma} \label{ultrac2}
For every topological space $X$, the following conditions are equivalent. \begin{enumerate}  
 \item Every sequence in $X$ converges.
\item  $X$ is   ultraconnected  and
 sequentially compact.
  \end{enumerate} 
 \end{lemma}

\begin{proof} 
(1) $\Rightarrow $  (2) is trivial, since
if every sequence in $X$ converges,
then  $X$ is  surely
sequentially compact.
Moreover, if $X$ is not ultraconnected, say 
$C_1, C_2 \subseteq X$ are closed and disjoint, 
then it is enough to consider  any sequence  
which takes infinitely many values in $C_1$ 
and infinitely many values in $C_2$, in order to get a nonconverging sequence.

In order to prove (2) $\Rightarrow $  (1),
 suppose that $X$ is  
 sequentially compact and ultraconnected, and let 
$(x_n) _{n \in \omega } $ be a sequence of elements of $X$.
By ultraconnectedness, $ \overline{\{ x_0 \}} \cap \dots \cap \overline{\{ x_n \}} \not= \emptyset  $, for every $ n \in \omega$.
For each $ n \in \omega$,
pick some   
$y_n \in  \overline{\{ x_0 \}} \cap \dots \cap \overline{\{ x_n \}}$.
By sequential compactness, 
some subsequence of 
$(y_n) _{n \in \omega } $ converges to some $y \in X$.
Then it is easy to see
that also 
$(x_n) _{n \in \omega } $ converges to $y $.
\end{proof}

Recall that  the \emph{splitting number} $\m s$ is
the least cardinality of a family $\mathcal S \subseteq [\omega]^ \omega $
such that, for every  $A \in   [\omega]^ \omega $,
there exists $S \in \mathcal S $ such that both
$A \cap S$ and $A \setminus S$ are infinite.  
See, e.~g., \cite{vd} for further details.

\begin{lemma} \label{ultra}
A product of $\geq \m s$ spaces which are not ultraconnected is not
sequentially compact.
 \end{lemma} 

\begin{proof} 
Let $X= \prod _{j \in J} X_j $ 
be a product of $\geq \m s$ spaces which are not ultraconnected.
Thus each $X_j$  has two disjoint closed nonempty subsets 
$C _j $  and $C'_j$. Letting $Y_j= C_j \cup C'_j$, for $j \in J$,
 we have that each $Y_j$ is a closed subsets of $X_j$,
hence   $Y= \prod _{j \in J} Y_j $   is a closed subset of 
 $X= \prod _{j \in J} X_j $.
For each $j \in J$, we can define a continuous surjective function from
$Y_j$ to the two elements discrete space 
$ \mathbf 2 = \{ d_1, d_2\} $,
by letting 
$f_j(y) = d_1$, for $y \in C_j$, and
$f_j(y) = d_2$, for $y \in C'_j$
(here we are using the assumption that
$C_j$  and $C'_j$ are disjoint and nonempty).
Naturally, we have a continuous surjective 
function from $Y= \prod _{j \in J} Y_j $ to 
$\mathbf 2 ^ \lambda $, for  $ \lambda =|J| \geq {\m s}$; 
however,
$\mathbf 2 ^{\m s}$ is well known not to be
sequentially compact
\cite[Theorem 6.1]{vd}, 
hence neither $Y$ nor
$X$ are sequentially compact.
\end{proof}

\begin{corollary} \label{seqcpseq} 
For every family $\mathcal T$ of topological spaces, the following conditions are equivalent. 
\begin{enumerate}
   \item 
All products of members of $\mathcal T$   are sequentially compact.
\item
For every $X \in \mathcal T$, $X^{\m s}$ is sequentially compact. 
\item
Every $X \in \mathcal T$ has the property that in $X$ every sequence converges.
\item
In all products of members of $\mathcal T$ every sequence converges.
  \end{enumerate}
\end{corollary}

 \begin{proof}
(1) $\Rightarrow $  (2) and (4) $\Rightarrow $  (1) are  trivial.

Had we replaced 
${\m s}$  with 
$2 ^ \omega $, in (2), 
the implication (2) $\Rightarrow $  (3) would follow immediately  by Corollary \ref{x}
using the 
  characterization of sequential compactness presented at the beginning, since
$| [\omega] ^ \omega|=2^ \omega $ and, for every $Z \in [\omega] ^ \omega $,  $F_Z$-compactness is trivially equivalent to $F$-compactness,
for the   Fr\'echet filter $F$ over $ \omega$. Then notice that
$F$-compactness
 is equivalent to the statement that
every sequence converges.

In order to prove (2) $\Rightarrow $  (3) 
for the improved bound ${\m s}$ in (2), notice that if (2) holds, 
then $X$ is ultraconnected,  by Lemma \ref{ultra}.
Since $X$ is trivially sequentially compact, we get that in $X$ 
every sequence converges, by
Lemma \ref{ultrac2}.

(3) $\Rightarrow $  (4)  By (3), every $X \in \mathcal T$ is  $F$-compact,
for the Fr\'echet filter $F$ over $ \omega$, and this implies (4), since
$F$-compactness is preserved under taking products
 \end{proof}

Notice that
the example $\mathcal T= \{ \mathbf 2  \} $
shows that 
the value $\m s$ 
in Corollary \ref{seqcpseq}(2) cannot be improved,
since $\mathbf 2  ^ \lambda $ is sequentially compact, 
for every $\lambda < \m s$ \cite[Theorem 6.1]{vd}.
We do not know whether 
sequential compactness is equivalent to 
sequencewise $\mathcal P$-\brfrt compactness, for some 
$\mathcal P$ with $|\mathcal P| < 2 ^ \omega $. 
Of course, if we could have $ |\mathcal P| = \m s$, then  the implication 
(2) $\Rightarrow $  (3) in Corollary \ref{seqcpseq}
would be a direct consequence of Theorem \ref{prodd}.
On the other hand, the equivalence of (1) and (2) in Theorem \ref{prodd}  
 shows that
 sequential compactness is not equivalent to 
sequencewise $\mathcal P$-\brfrt compactness, for some 
$\mathcal P$ with $|\mathcal P| < \m s $, since, again, 
 $\mathcal K= \{ \mathbf 2  \} $ would give a counterexample.

\section{Concluding remarks} \label{concl}

The problem of deciding whether, in some class $\mathcal K$, 
sequencewise $\mathcal P$-\brfrt compactness  and sequencewise  $\mathcal P'$-compactness  are equivalent,
for certain families $\mathcal P$ and $\mathcal P'$, 
 might involve very difficult
problems, sometimes of purely set-theoretical nature
(even just when $\mathcal K$ is the class of all topological spaces).
Other particularly interesting cases are given by the class of topological groups,
and the class of normal spaces.

For a class $\mathcal K$ of topological spaces, and for $F$, $G$ filters
(not necessarily over the same set),
define the following (pre-)order: $F \leq _{C, \mathcal K} G$ if and only if
every $G$-compact topological space in $\mathcal K$ is 
$F$-compact. Strictly speaking, $\leq _{C, \mathcal K}$ 
is not an order relation, but it induces an order on the equivalence classes
modulo the relation $\equiv _{C, \mathcal K} $ defined by $F \equiv _{C, \mathcal K} G$ if and only if both 
$F \leq _{C, \mathcal K} G$ and $G \leq _{C, \mathcal K} F$.
When $\mathcal K$ is the class of all 
 Tychonoff spaces, and $F$ and $G$ are ultrafilters,
$  \leq _{C, \mathcal K}$ is the \emph{Comfort (pre-)order}. 
See \cite{filo} for a survey, in particular
Section 3.

Trivially, if sequencewise $\mathcal P$-\brfrt compactness is equivalent to 
$F$-\brfrt compactness in
$\mathcal K$,
then    $F \leq _{C, \mathcal K} G$,
for every $G \in \mathcal P$.
Hence, 
by Theorem \ref{prodd}(1) $\Leftrightarrow $  (3), 
if $\mathcal P$ is a class of filters 
with no minimum
 with respect to  $\leq _{C, \mathcal K} $ (here, ``minimum''  is intended modulo equivalence), 
then  
sequencewise $\mathcal P$-\brfrt compactness is not
preserved under taking products of spaces in $\mathcal K$.

However, the existence of such a minimum in $\mathcal P$ is not 
a sufficient condition for preservation under taking products,
as already the example of sequential compactness shows
(indeed, $F_Z \equiv _{C, \mathcal K} F _{Z'} $,
 for every $Z, Z' \in [\omega] ^{ \omega } $ and every $\mathcal K$). 
Using a result
by Garc{\'{\i}}a-Ferreira \cite{GFczech}, we can give an example in which all members of $\mathcal P$  are ultrafilters. For every non principal
ultrafilter $D$ over $ \omega$, \cite[Example 2.3]{GFczech} constructs a space which is not $D$-compact, but which is
sequencewise $\mathcal P$-\brfrt compact, where
$\mathcal P$ is the family of all  ultrafilters over $ \omega$ 
which are 
Rudin-Keisler equivalent to $D$ 
($\mathcal P$ is called $T _{RK}(D) $ in \cite{GFczech}).  
A fortiori, all elements of $\mathcal P$ are Comfort
equivalent, hence each of them is a minimum in $\mathcal P$, under equivalence.
However, sequencewise $\mathcal P$-\brfrt compactness is not preserved
under products, since, otherwise, by Theorem \ref{prodd},
it
would be equivalent to $D'$-compactness, for some $D' \in\mathcal P$.  
But $D'$-compactness is equivalent to $D$-compactness,
hence $\mathcal P$-\brfrt compactness 
would be equivalent to $D$-compactness,
and this is false, as shown by the space constructed by 
Garc{\'{\i}}a-Ferreira.

Of course, in any single particular application of 
Theorem \ref{prodd}, the needed results can be proved directly.
Nevertheless, we believe that the theory presented here 
has some intrinsic interest.
At the very least, it
has the advantage
of presenting many distinct results in an unified way,
 providing a conceptual clarification.

Apart from this,
 the notion of sequencewise $\mathcal P$-\brfrt compactness
is of some use in at least two respects.
First, we have showed that, when studying 
the satisfiability
 of 
a topological property
in products, it is convenient  to translate this property
(if possible) in terms of sequencewise $\mathcal P$-\brfrt compactness.
This provides a standard method for dealing with the problem,
and is what 
Bernstein 
\cite{Be},
Ginsburg and Saks \cite{GS}, Saks \cite{Sa} and
Caicedo \cite{Ca}, among others, have done in particular cases, 
as we mentioned in the introduction.
We have continued this line of research
 in \cite{ufmeng} 
for  
the Menger and  the Rothberger properties,
 here for sequential compactness,
and in \cite{prodlm} for $[\mu, \lambda ]$-compactness.
See also \cite{genfrol}.

 Second, the theory presented here naturally leads to new problems.
For example, it stresses
the importance of studying when
 sequencewise $\mathcal P$-\brfrt compactness and
sequencewise $\mathcal P'$-\brfrt compactness are equivalent,
for different $\mathcal P$ and $\mathcal P'$.
This should be particularly interesting
when restricted to special classes $\mathcal K$ of spaces.  
Just to present a simply stated  but intriguing case, 
really little is known
about Comfort order restricted to the class $\mathcal K$ of normal spaces,
that is  $ \leq _{C, \mathcal K} $.
Notice that this is just
a particular case in which both $\mathcal P$ and
$\mathcal P'$ are singletons.

Another kind of problems arise as follows.
So far, we have considered certain given topological properties,
and showed that there exists an appropriate family
$\mathcal P$ which characterizes them
in terms of sequencewise $\mathcal P$-\brfrt compactness.
One can also try to 
follow
the other direction, that is, take some interesting
classes of ultrafilters, and 
consider the associated topological properties.

We can also introduce  a ``pseudocompact-like''  version of
sequencewise $\mathcal P$-\brfrt compactness.
If $X$ is a topological space, 
$I$ is a set, $(Y_i) _{i \in I} $ 
is an $I$-indexed sequence  of subsets of  $X$,
and $F$  is a filter over $I$, 
 a point $x \in X$ is an
 \emph{$F$-limit point} 
(Choquet \cite{Ch})
 of the sequence
$(Y_i) _{i \in I} $
 if 
 $\{ i \in I \mid Y_i \cap U  \not= \emptyset   \} \in F$,
for every open neighborhood $U$ of $x$.
If $\mathcal P$ is a family of filters over $I$,
we say that $X$ is \emph{sequencewise $\mathcal P$-\brfrt pseudocompact}
 if, for every $I$-indexed sequence of nonempty open subsets of $X$,
there is $F \in \mathcal P$  such that the sequence 
has an $F$-limit point. Examples include 
pseudocompactness and $D$-pseudocompactness 
\cite{GS}; see \cite{tapp2,tproc2} and \cite[Sections 4 and 5]{cmuc}
for further examples. 
Sometimes the above examples are presented in
equivalent formulations; however, they can 
be recast in terms of sequencewise 
$\mathcal P$-\brfrt pseudocompactness by a remark  analogous
to \cite[Remark 5.4]{cmuc}.
Though the study of the behavior of sequencewise $\mathcal P$-\brfrt pseudocompactness
with respect to products goes beyond the scope of the present note,
 let us notice that, in general, results about sequencewise $\mathcal P$-\brfrt compactness
do not necessarily generalize, as they stand. 
As a classical example, 
products of
countably compact spaces and of pseudocompact spaces
behave in a different way with respect to ultrafilter convergence
\cite[Theorem 2.6 and Example 4.4]{GS}. 
More elaborate examples (and a possible explanation for  the asymmetry) can be found in  
\cite[Section 5]{tproc2} and \cite[Section 4]{tapp2}.

As a final remark, let us mention that all the results of the present note 
can be easily generalized 
to the case of  $\kappa$-box products,
provided that we consider only 
$\kappa$-complete filters and ultrafilters.
The $\kappa$-box product $\Box ^{ \kappa }  _{j \in J}  X_j $
is defined on the set 
$\prod _{j \in J}  X_j $,
and its topology has 
$\{ \prod _{j \in J}  O_j \mid O_j \text{ is open in } X_j, \text{ and } 
|\{ j \in J \mid O_j \not= X_j\}| < \kappa  \}$
as a base.

\section*{Acknowledgements} 

We thank anonymous referees of 
\cite{cmuc}
 for many helpful comments which have been of great use
in clarifying matters related both to 
\cite{cmuc}
 and to the present work.

We thank our students from Tor Vergata University for stimulating questions.

\end{document}